\documentclass[reqno,12pt,a4paper]{amsart}

\usepackage[utf8]{inputenc}

\usepackage{enumerate}
\usepackage{amstext,amsmath,amsthm,amsfonts,amssymb,amscd}
\usepackage{latexsym,mathrsfs,dsfont,euscript}
\usepackage{fullpage}
\usepackage{color}

\usepackage{hyperref}
\hypersetup{
  colorlinks,%
  citecolor=blue,%
    filecolor=red,%
    linkcolor=red,%
    urlcolor=blue
}

\usepackage{graphicx}

\newtheorem{theorem}{Theorem}[section]
\newtheorem{proposition}[theorem]{Proposition}
\newtheorem{lemma}[theorem]{Lemma}

\theoremstyle{definition}
\newtheorem{definition}{Definition}[section]

\theoremstyle{remark}

\numberwithin{equation}{section}

\newcommand{\abs}[1]{\left\vert#1\right\vert}

\allowdisplaybreaks
\begin{document}
\title[]{On the evolution of topological connectivity by thresholding of affinities. An application to public transport}
%


\author[]{Hugo Aimar}
\author[]{Carlos Exequiel Arias}
\author[]{Ivana G\'{o}mez}
%


\begin{abstract}
In this paper we use the neighborhood topology generated by affinities between pairs of points in a set, in orden to explore the underlying dynamics of connectivity by thresholding of the affinity. We apply the method to the connectivity  provided by the public transport system in Buenos Aires.
\end{abstract}

\maketitle

\section{Introduction}\label{sec:intro}
The number of connected components of a space is a topological concept which provides some interesting information on the interactions between the points of the set. On the other hand, the evolution of the number of connected components by thresholding of affinities between points or parts of a set, can be taken as a benchmark of the system under analysis. Roughly speaking an affinity can be thought as the reciprocal of a distance. Nevertheless the mathematical notion of metric or distance on a set is very precise and in applications, frequently, we find non metrizable affinities. Before introducing the definitions, let us consider the weighted undirected affinity graph to which we shall apply our results.

Approximately one third of the total population of Argentina is concentrated around Buenos Aires City. The acronym AMBA (Área Metropolitana Buenos Aires) is used to name these 41 cities with a total population of 17 million people. The public transportation of people between these 41 districts has a unified system of electronic tickets, for buses, trains, subways. The acronym of this system is SUBE (Sistema Único de Boleto Electrónico) Unified System for Electronic Tickets. The global public data provided by SUBE, regarding the number of daily transactions, gives a transportation affinity matrix between any pair of the 41 cities in AMBA. As we shall see this affinity is far from being related to the geographic distance of any couple of districts. In \cite{AcAiGoMoTCAM22} the authors introduce in this setting the diffusive metrics of Coifman and Lafon \cite{CoifmanLafon06} defined by the spectrum of the Laplace operator in the weighted undirected graph of the 41 cities as vertices and a measure given by the SUBE data of the affinity of any two cities in AMBA. We shall consider three of the these matrices corresponding to 2020 in three diferent moments of the evolution of COVID-19 in Argentina: March, April and June and we shall compare their connection dynamics by thresholding of the affinities. 

The aim of this paper is to construct non necessarily metrizable topologies on data sets with affinities and to study the evolution of the number of connected components after thresholding of the affinities. Once the theoretical aspects of the proposal are achieved in sections \ref{sec:AffinitiesandInducedTopologiesonAbstractSets} and \ref{sec:ThresholdingOfAffinities}, in Section~\ref{sec:AnApplication} we apply the results to the SUBE matrices introduced above.

\section{Affinities and induced topologies on abstract sets}\label{sec:AffinitiesandInducedTopologiesonAbstractSets}

Given a metric on a set $X$, a topology is immediately determined on $X$ as the family of all open subsets of $X$, i.e. the sets $G\subseteq X$ such that for every $x\in G$ there exists a positive $r$ for which the metric ball centered at $x$ with radious $r$ is contained in $G$. The notion of affinity is, in some heuristic sense, reciprocal to the notion of metric. In fact, two points in the abstract metric space $X$, which could be two vertices of a metric graph, can be considered to have high affinity if they are close in the metric sense. Let us give the precise definition of what we shall call an affinity on an abstract set.
\begin{definition}
	Let $X$  be a set. A positive real valued function $A$ defined on the product set $X\times X$ is said to be an \textbf{affinity on $X$} if\begin{enumerate}
		\item $A$ is symmetric; $A(x_1,x_2)=A(x_2,x_1)$ for every $x_1$, $x_2\in X$; and

		\item $A(x,x)=+\infty$ for every $x\in X$.
	\end{enumerate}
We shall also say that $(X,A)$ is an affinity space.
\end{definition}

In \cite{AiGoAGMS18} two of the authors show that a natural metric determined by the affinity $A$ can be defined on $X$ if and only if the affinity $A$ satisfies a transitivity property. This transitivity property reflects an heuristic idea, which is not always true in some interesting models as we shall see in our application in the last section.

This heuristic idea, which sometimes is not valid in concrete situations, is the following: $A(x_1,x_2)$ large and $A(x_2,x_3)$ large implies that $A(x_1,x_3)$ is large. Or more quantitatively $A(x_1,x_2)>\lambda>0$ and $A(x_2,x_3)>\lambda$, then $A(x_1,x_3)>\lambda/2$. Actually the precise result is that under some extra transitivity condition in $A$ there exists a metric $d$ on $X$ related to $A$ by $A(x_1,x_2)=\varphi(d(x_1,x_2))$ with $\varphi(0)=+\infty$, $\varphi(\infty)=0$ and $\varphi$ decreasing. (See \cite{AiGoAGMS18}).

The aim of this section is to construct a topology associated to a given non transitive affinity without going through metric structures on the given set. The main tool is the construction of a topology on $X$ by neighborhood systems defined on $X$. See Kelley~\cite{KelleybookEudeba}. Let us recall that a topology on $X$ is a subfamily $\tau$ of subsets of $X$ that contains both $\emptyset$ and $X$ and is closed under finite intersections and arbitrary unions. With $\mathcal{P}(X)$ we denote, as usual, the set of all subsets of $X$.

\begin{proposition}\label{propo:topologyNxfamily}
	Let $X$ be a set and let $\mathcal{N}: X\to \mathcal{P}(\mathcal{P}(X))$ be a function that to each $x$ assigns a nonempty family $\mathcal{N}_x$ of subsets of $X$ satisfying the following properties,
	\begin{enumerate}[(i)]
		\item if $U\in\mathcal{N}_x$, then $x\in U$;
		\item if $U$ and $V$ belong to $\mathcal{N}_x$, then $U\cap V\in\mathcal{N}_x$;
		\item if $U\in\mathcal{N}_x$ and $V\supset U$, then $V\in\mathcal{N}_x$.
	\end{enumerate}
Then the family
\begin{equation*}
	\tau=\left\{U: U\in\mathcal{N}_x \textrm{ for every } x\in U\right\}
\end{equation*}
is a topology in $X$.
\end{proposition}
\begin{proof}
	Trivially $\emptyset$ and $X$, both belong to $\tau$, since, by \textit{(iii)} $X\in\mathcal{N}_x$ for every $x\in X$. If $U$ and $V\in \tau$, then, by \textit{(ii)}, $U\cup V\in\tau$. From \textit{(iii)} it follows also that $\cup_{\alpha\in\Gamma} U_\alpha\in\tau$ if each $U_\alpha$ belongs to $\tau$.
\end{proof}

Connectivity is a topological property. As such it can be defined precisely in any topological space $(X,\tau)$.
\begin{definition}\label{def:connectedspace}
	Let $(X,\tau)$ be a topological space, i.e. a set $X$ with a topology $\tau$. We say that $(X,\tau)$ is \textbf{connected} if $X$ can not be decomposed as the union of two nonempty and disjoint open sets. In other words $(X,\tau)$ is connected if do not exist $U$ and $V$ in $\tau$ with
	\begin{enumerate}[(\textrm{C}1)]
		\item $U\neq \emptyset$, $V\neq \emptyset$;
		\item $U\cap V=\emptyset$; and,
		\item $U\cup V= X$.
	\end{enumerate}
\end{definition}
Given a subset $S$ of $X$ when $(X,\tau)$ is a topological space, the inherited topology in $S$ is given by $\tau_S=\{U\cap S: U\in\tau\}$. We say that a subset $S$ of $X$ is connected if the topological space $(S,\tau_S)$ is connected with the above definition. This can also be rephrased with no reference to the hereditary topology as follows. If $(X,\tau)$ is a topological space, a set $S\subset X$ is connected if and only if do not exist two open sets $U$ and $V$ in $\tau$ such that $S\subset U\cup V$, $U\cap V=\emptyset$, $U\cap S\neq \emptyset$ and $V\cap S\neq \emptyset$. 

\begin{lemma}\label{lem:unionconnected}
	Let $(X,\tau)$ be a topological space and let $\{K_\alpha: \alpha\in\Lambda\}$ be a family of connected sets in $X$, such that $\bigcap_{\alpha\in\Lambda}K_\alpha\neq \emptyset$. Then $\bigcup_{\alpha\in\Lambda}K_\alpha$ is connected.
\end{lemma}
\begin{proof}
	Take $x\in \bigcap_{\alpha\in\Lambda}K_\alpha$. If $K=\bigcup_{\alpha\in\Lambda}K_\alpha$ were not connected, we would have $U$ and $V$ open and disjoint members of $\tau$ such that $U\cap K\neq \emptyset$, $V\cap K\neq \emptyset$ and $K\subset U\cup V$. Since $x\in K$ and $U\cap V=\emptyset$, then $x$ belongs to one and only one of them. Assume that $x\in U$, so $x\notin V$. Since $V\cap K\neq \emptyset$ then $V\cap K_\alpha\neq\emptyset$ for some $\alpha\in\Lambda$. Hence, $K_\alpha\subset K\subset U\cup V$, but $x\in U\cap K_\alpha$, so that $U\cap K_\alpha\neq\emptyset$. This implies that $K_\alpha$ is not connected.
\end{proof}
For a given topological space $(X,\tau)$ and a given point $x\in X$ the family $C_x=\{S\subset X: x\in S \textrm { and } S \textrm{ is connected}\}$ is nonempty. From Lemma~\ref{lem:unionconnected}, we have that $C(x)=\bigcup_{\{C \textrm{ is connected and } x\in C\}}C$ is connected for every $x\in X$. $C(x)$ is named the \textbf{connected component} of $(X,\tau)$ containing $x$. The whole space is connected if and only if there is one and only one connected component. The number of connected components of $(X,\tau)$ shall be an important parameter for our later analysis. We denote it by $\kappa(X,\tau)$ or $\kappa$ when the topological context is clear.

Let us now use Proposition~\ref{propo:topologyNxfamily} in order to construct a topology on $X$ starting from an affinity in $X$.

\begin{proposition}\label{propo:topologyAffinity}
	Let $(X,A)$ be an affinity space. For $\alpha>0$ and $x\in X$, set $E(x,\alpha)=\{y\in X: A(x,y)>\alpha\}$. Let $\mathcal{N}$ be the function that to each $x\in X$ assigns the family
	\begin{equation*}
		\mathcal{N}_x=\left\{U\subset X: E(x,\alpha)\subset U \textrm{ for some } \alpha>0\right\}
	\end{equation*}
of parts of $X$. Then, the function $\mathcal{N}$ satisfies \textit{(i)}, \textit{(ii)} and \textit{(iii)} in Proposition~\ref{propo:topologyNxfamily}. Hence the family
\begin{equation*}
	\tau_A = \left\{U\subset X: \textrm{ there is a function } \alpha: U\to\mathbb{R}^+ \textrm{ such that } E(x,\alpha(x))\subset U \textrm{ for every }x\in U \right\},
\end{equation*}
is a topology on $X$.
\end{proposition}
\begin{proof}
	It is clear from Proposition~\ref{propo:topologyNxfamily} that we only need to check that $\mathcal{N}$ satisfies \textit{(i)}, \textit{(ii)} and \textit{(iii)}. Since $A(x,x)=+\infty$, we have that $x\in E(x,\alpha)$ for every $\alpha>0$. Hence if $U\in\mathcal{N}_x$, we have that $x\in U$ because $E(x,\alpha)\subset U$ for some $\alpha>0$. This proves  \textit{(i)}. Take now $U$ and $V$ two sets in $\mathcal{N}_x$, then, there exists $\alpha_U$ and $\alpha_V$ two positive real numbers with $E(x,\alpha_U)\subset U$ and $E(x,\alpha_V)\subset V$. Hence, with $\alpha=\max\{\alpha_U,\alpha_V\}$ we have that $E(x,\alpha)\subset U\cap V$ and $U\cap V\in\mathcal{N}_x$. So \textit{(ii)} is proved. Property \textit{(iii)} is immediate.
\end{proof}

Some examples are in order. Notice first that $A_1\equiv +\infty$ is an affinity in any set $X$, then the induced topology $\tau_{A_1}$ is the trivial $\{\emptyset,X\}$. Hence, the topologies provided by Proposition~\ref{propo:topologyAffinity} need not to be metrizable. On the other extreme case, if $A_2(x,x)=+\infty$ and $A_2(x_1,x_2)=1$ for $x_1\neq x_2$, since $\{x\}=E(x,2)$ for every $x\in X$, then $\{x\}$ is open (belongs to $\tau_{A_2}$) for every $x\in X$. So that $\tau_{A_2}=\mathcal{P}(X)$. Notice also that while $(X,\tau_{A_1})$ is connected, $(X,\tau_{A_2})$ is totally disconnected in the sense that the only connected subsets of $X$ are those that contain only one point. It is easy to check that when $(X,d)$ is a metric space, then $A(x_1,x_2)=d^{-1}(x_1,x_2)$ is an affinity on $X$ and the topology $\tau_A$ is the metric topology in $X$. Our example $A_1$ above shows that not every affinity induced topology is metrizable. 

\section{Thresholding of affinities}\label{sec:ThresholdingOfAffinities}
Let $(X,A)$ be a given affinity space. For $\lambda$ positive define the thresholding of $A$ at $\lambda$, by
\begin{equation*}
	A^{\lambda}(x,y)= \begin{cases} 
		+\infty	&\text{ for } A(x,y)>\lambda \\
		A(x,y)	&\text{ for }  A(x,y)\leq \lambda. 
	\end{cases} 
\end{equation*}
Some elementary properties of the family of affinity spaces $(X,A^\lambda)$ are contained in the next statement.
\begin{proposition}\label{propo:propertiesfamilyofaffinityspaces}
\,
	
	\begin{enumerate}[(a)]
		\item For each $\lambda >0$, 	$(X,A^{\lambda})$ is an affinity space;
		\item for $0<\lambda_1<\lambda_2$, $A^{\lambda_1}(x,y)\geq A^{\lambda_2}(x,y)\geq A(x,y)$.
	\end{enumerate}
\end{proposition}
In this way, given an affinity $A$ on the set $X$, we produce a one parameter family of affinities $A^\lambda$ on $X$ for $\lambda>0$. Applying Proposition~\ref{propo:topologyAffinity} for each $\lambda>0$ we obtain a corresponding topology $\tau_\lambda$ on $X$ determined by the affinity space $(X,A^{\lambda})$. Let us denote with $\tau_A$ the topology induced on $X$ by $A$. With the notation introduced in Section~\ref{sec:AffinitiesandInducedTopologiesonAbstractSets}, for each topological space $(X,\tau)$, $\kappa(X,\tau)=\#\{C: C \textrm{ is a connected component of }(X,\tau)\}$, where $\#$ denotes the cardinal. With the above family of topologies on $X$ given by the thresholding of $A$, we obtain a function $\kappa:\mathbb{R}^+\to [1,\#(X)]$ given by $\kappa(\lambda)=\kappa(X,\tau_\lambda)$. Set also $\kappa_A=\kappa(X,\tau_A)$.

\begin{proposition}
	With the above notation, the following properties hold
	\begin{enumerate}[(a)]
		\item if $0<\lambda_1<\lambda_2$, then $\tau_{\lambda_1}\subseteq \tau_{\lambda_2}\subseteq\tau_A$;
		\item $\kappa(\lambda)$ is nondecreasing and bounded above by $\kappa_A$;
		\item when $X$ is finite, for $\lambda$ large enough, we have that $\tau_\lambda=\tau_A$ and $\kappa(\lambda)=\kappa_A$.
	\end{enumerate}
\end{proposition}
\begin{proof}
	\textit{(a)} Take $0<\lambda_1<\lambda_2$, $U\in\tau_{\lambda_1}$ and $x\in U$. Then, there exists $\alpha>0$ such that $E_{\lambda_1}(x,\alpha)=\{y\in X: A^{\lambda_1}(x,y)>\alpha\}\subset U$. On the other hand, since from \textit{(b)} in Proposition~\ref{propo:propertiesfamilyofaffinityspaces}, $A^{\lambda_1}(x,y)\geq A^{\lambda_2}(x,y)$, we have that $E_{\lambda_2}(x,\alpha)\subseteq E_{\lambda_1}(x,\alpha)$. So that $E_{\lambda_2}(x,\alpha)\subset U$ and $U\in\tau_{\lambda_2}$. The same argument shows that $\tau_{\lambda_2}\subseteq \tau_{A}$ for every $\lambda_2>0$. \textit{(b)} Take again $0<\lambda_1<\lambda_2$. Then, since $\tau_{\lambda_1}\subseteq \tau_{\lambda_2}$, every $\tau_{\lambda_2}$ connected set is also $\tau_{\lambda_1}$ connected. In fact, if $K$ is $\tau_{\lambda_2}$ connected, then $K$ can not be written as $K=(K\cap U)\cup (K\cap V)$ with $U, V\in\tau_{\lambda_2}$, $U\cap V=\emptyset$ and $K\cap U\neq \emptyset\neq K\cap V$. Since every $U$ in $\tau_{\lambda_1}$ is also in $\tau_{\lambda_2}$, then $K$ can not be $\tau_{\lambda_2}$ connected. Take now $C_2(x)$ the connected component in $\tau_{\lambda_2}$ containing $x\in X$, hence $C_2(x)$ is also $\tau_{\lambda_1}$ connected and contains $x$. Thus $C_1(x)=\bigcup_{\{C: \tau_{\lambda_1} \textrm{ connected and } x\in C\}} C \supseteq C_2(x)$. Then we have a function assigning to each $\tau_{\lambda_2}$ connected component $C_2$ a $\tau_{\lambda_1}$ connected component $C_1$ which is onto. Hence $\kappa(\lambda_2)\geq \kappa(\lambda_1)$. The same argument shows that $\kappa(\lambda)\leq \kappa_A$ for every $\lambda>0$. \textit{(c)} Notice that, being 
	$\{A(x,y): A(x,y)<\infty\}$ a finite set of positive real numbers, it has a maximum $M$. Hence, for $\lambda >M$ we get $A^{\lambda}=A$. So that $\tau_{\lambda}=\tau_A$ and $\kappa(\lambda)=\kappa_A$.
\end{proof}

Let $A$ be a given affinity on a finite set $X=\{1,2,\ldots,n\}$. In this case $A$ can be seen as an $n\times n$ symmetric matrix with positives entries such that $A_{ii}=+\infty$ for every $i=1,\ldots,n$. Of course $A_{ij}$ can take the value $+\infty$ in several other pairs $(i,j)$ outside the diagonal as far as the symmetry is preserved. We proceed to construct a simple graph $\mathcal{G}_A$ associated to $A$ as follows. Take $\mathcal{V}=X$ as the set of vertices and $\mathcal{E}_A=\{\{i,j\}: A_{ij}=A_{ji}=+\infty\}$ as the set of edges; $\mathcal{G}_A=(\mathcal{V},\mathcal{E}_A)$.

The adjacency matrix of this graph is $\mathfrak{A}_{ij}=1$ if $\{i,j\}\in\mathcal{E}_A$ and $\mathfrak{A}_{ij}=0$ when $\{i,j\}\notin\mathcal{E}_A$. In other words, $\mathfrak{A}$ has ones where $A$ takes the value $+\infty$ and $\mathfrak{A}$ has zeros where $A$ has finite numbers.

Given a simple graph $\mathcal{G}=(\mathcal{V},\mathcal{E})$ the idea of a path joining two vertices $i, j$ is simple and well known. We say that there is a \textbf{path joining $i$ and $j$} if there exist an integer $k>1$ and a sequence $i=i_1,i_2,\ldots,i_k=j$ of vertices in $\mathcal{V}$ such that $\{i_l,i_{l+1}\}$ belongs to $\mathcal{E}$ for every $l=1,\ldots,k-1$. A subset $\mathcal{V}^{'}$ of $\mathcal{V}$ is said to be \textbf{path connected} with respect to $\mathcal{G}$ if for every pair $i$ and $j$ of vertices in $\mathcal{V}^{'}$ there exists a path in $\mathcal{V}^{'}$ joining $i$ and $j$. A \textbf{path connected component} is a maximal path connected subset of $\mathcal{V}$. In other words $\mathcal{V}^{*}$ is a path connected component if $\mathcal{V}^{*}$ is path connected and if $\widetilde{\mathcal{V}}\supsetneqq\mathcal{V}^{*}$, then $\widetilde{\mathcal{V}}$ is not path connected.

In order to use \texttt{Python} algorithms, such as \texttt{number\_connected\_components}, the following result will be useful.

\begin{proposition}
	Let $A$ be a given affinity on the set $X=\{1,2,\ldots,n\}$ and let $\mathcal{G}_A$ be the induced simple graph. Then
	\begin{enumerate}[(i)]
		\item a subset $\widetilde{X}$ of $X$ is connected in the sense of Definition~\ref{def:connectedspace} with respect to the neighborhood topology $\tau_A$ in $X$ if and only in $\widetilde{X}$ is path connected with respect to $\mathcal{G}_A$;
		\item the number of topological connected components in $(X,\tau_A)$ coincides with the number of path connected components with respect to $\mathcal{G}_A$.
	\end{enumerate}
\end{proposition}
\begin{proof}
	Notice first that \textit{(ii)} follows from \textit{(i)}. On the other hand, to prove \textit{(i)} we only need to show that the entry $\mathfrak{A}_{ij}=1$ if and only if any $\tau_A$ open set containing $i$ also contains $j$. But $\mathfrak{A}_{ij}=1$ if and only if $A_{ij}=+\infty$. Take $U$ a $\tau_A$ open set such that $i\in U$. Then there exists $\alpha>0$ such that $E(i,\alpha)\subset U$. Since $E(i,\alpha)=\{k:A_{ik}>\alpha\}$ and $A_{ij}=+\infty$ we have that $j\in E(i,\alpha)\subset U$. So that $j\in U$. On the other hand, if $j$ belongs to any open set containing $i$, then $j\in E(i,\alpha)$ for every $\alpha>0$. This implies $A_{ij}=+\infty$ and $\mathfrak{A}_{ij}=1$.
\end{proof}

Before introducing, in the next section, the application of the above analysis, let us illustrate some simple examples in the real line in order to get some insight regarding the relation between the shape of the curve $\kappa(\lambda)$ as a function of the threshold parameter and the distribution of the data points. The four examples are subsets of real numbers and the affinity is the metric one, that is $A(x,y)=\frac{1}{\abs{x-y}}$. For the first example, see Figure~\ref{fig:allSeriesConnnectedComponents}, consider the set $X_1=\left\{f_1(i)=\log_2 i: i=1,2,\ldots,20\right\}$. For the second, we take the set $X_2=\left\{f_2(i)=\sqrt{i-1}: i=1,2,\ldots,20\right\}$, for the third $X_3=\left\{f_3(i)=20\bigl(1-\frac{1}{i}\bigr): i=1,2,\ldots,20\right\}$ and for the last one we take $X_4=\left\{f_4(i)=20\right.\bigl(1-\left(\frac{5}{6}\right)^{i-1}\bigr): i=1,2,\ldots,20\left.\right\}$. Let us notice that the concave shapes of $\kappa(\lambda)$ correspond to convergence of the generating sequences as in cases $X_3$ and $X_4$.
%
%
%
\begin{figure}[!]
	\begin{tabular}{|c|c|}
		\hline
		$f_m$ & $\kappa(X_m,\tau_A)$ for $m=1,2,3,4$ \\
		\hline
		\includegraphics[width=7cm]{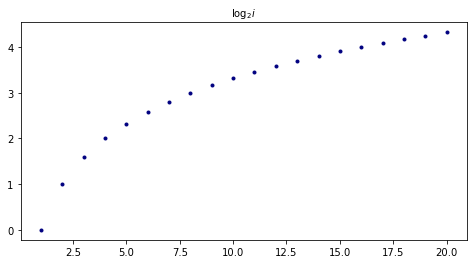} & \includegraphics[width=7cm]{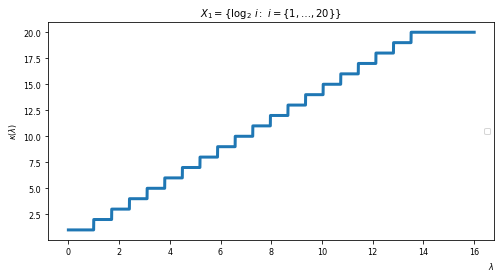}  \\
		\hline
		\includegraphics[width=7cm]{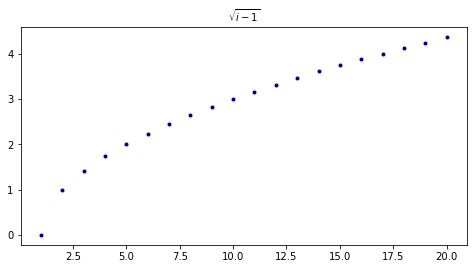} &
		\includegraphics[width=7cm]{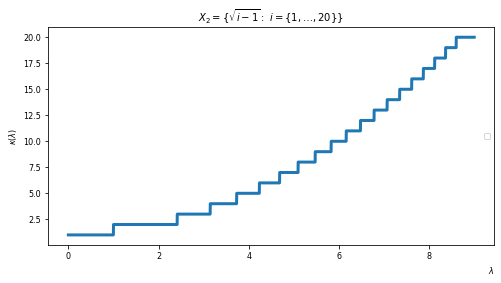}  \\
		\hline
		\includegraphics[width=7cm]{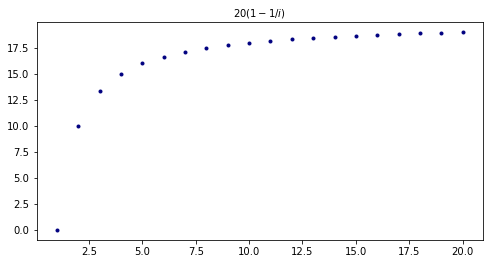} &
		\includegraphics[width=7cm]{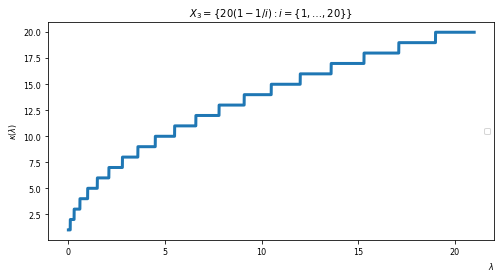}  \\
		\hline
		\includegraphics[width=7cm]{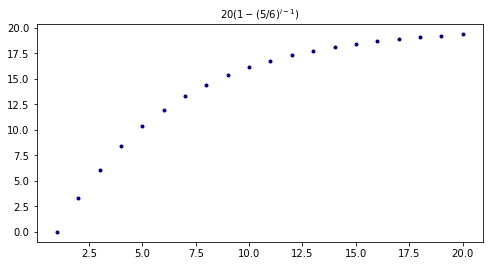} &
		\includegraphics[width=7cm]{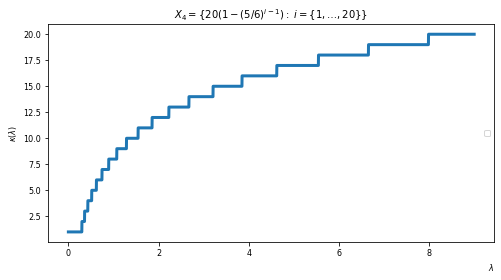}  \\
		\hline
	\end{tabular}
	\caption{$\kappa(\lambda)$ (on the right) for different distribution data points (on the left).}\label{fig:allSeriesConnnectedComponents}
\end{figure}

\section{Application to the public transport in Buenos Aires}\label{sec:AnApplication}
In Figure~\ref{fig:AMBAmap}, we show schematically the relative geographic locations of the 41 cities in AMBA. The numbers in each one of the districts will be used as the indices $i=1,\ldots,41$ that identify the vertices of the graph and the affinities that we shall consider.
\begin{figure}[!]
	\includegraphics[width=11cm]{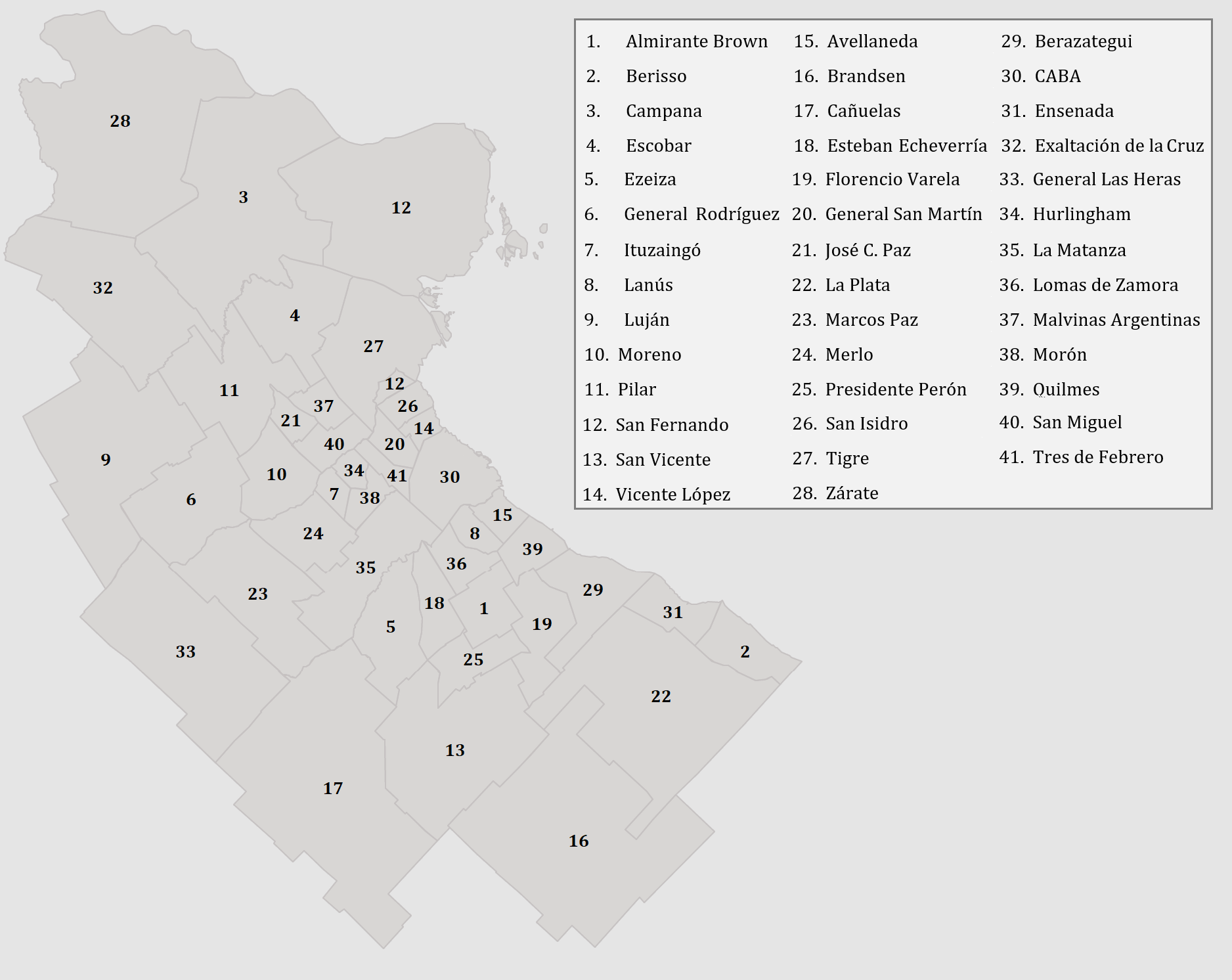}
	\caption{Map of the 41 cities of AMBA}\label{fig:AMBAmap}
\end{figure}

The normalized affinity matrices that we shall consider correspond to three different days of 2020 at the beginning of the circulation of COVID-19 in Argentina. The first $A^1$, corresponds to March~4th 2020 before the detection of the spread of COVID-19, with the public transport system working normally. The second, $A^2$, corresponds to April~8th 2020, with the beginning of the circulation restrictions imposed by the health administration. The third, $A^3$, is constructed with the data corresponding to June~3rd 2020, when the circulation restrictions started to weaken. 
In all these cases we shall apply the algorithm introduced in the previous sections and we shall obtain the corresponding functions $\kappa(\lambda)$ that describe the evolution of the number of connected components as a function of the threshold parameter $\lambda$ in the three situations described.

The matrix $A^1$, corresponding to March 2020, is the $41\times 41$ normalized matrix with $A^1_{ij}$ measuring the intensity (number of passengers) going from city $i$ to city $j$ and viceversa by March~4th 2020, is given by Figure~\ref{fig:AffinityMatrixMarch}.
\begin{figure}[b]
	\includegraphics[width=11cm]{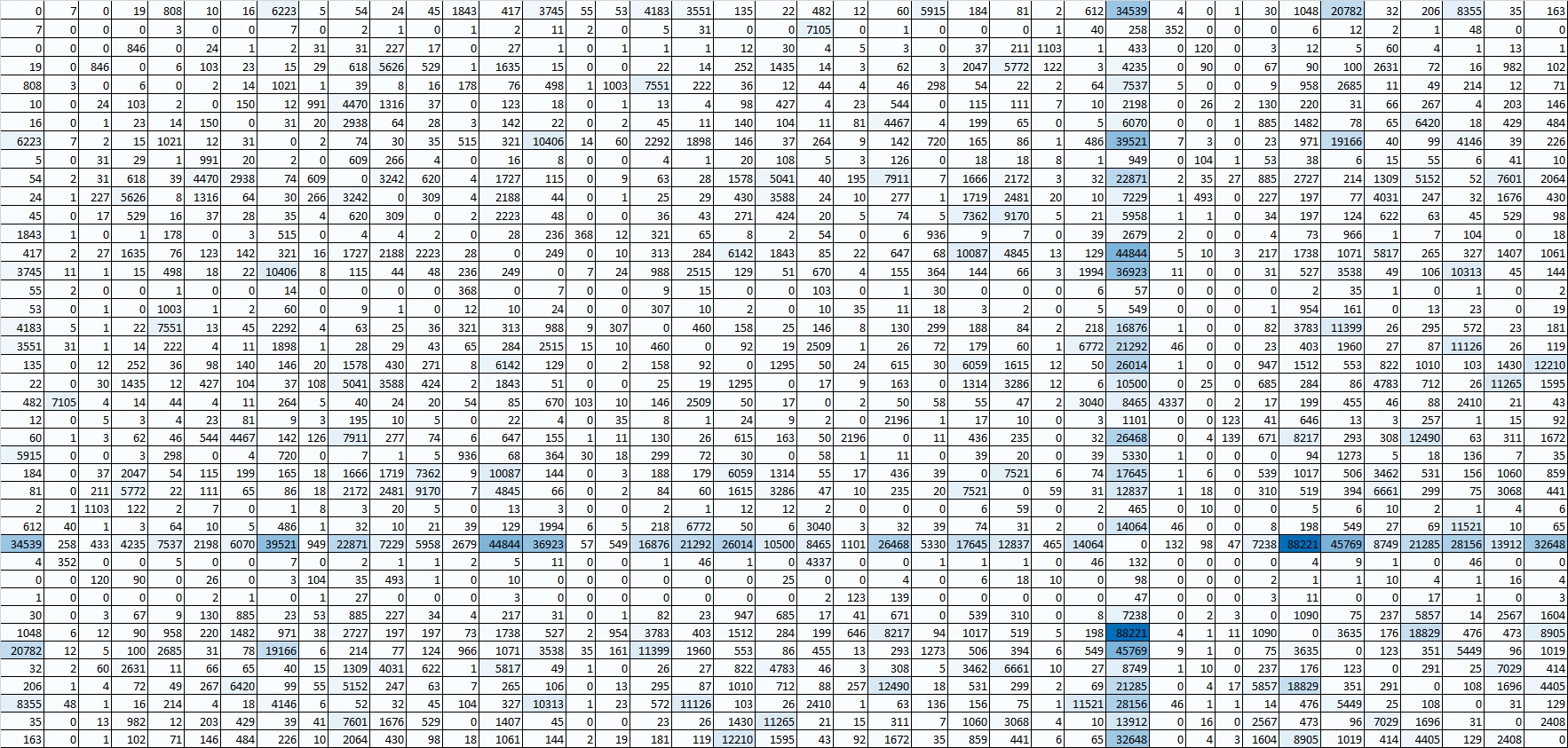}
	\caption{Unnormalized affinity matrix $A^1$ corresponding to March 2020}\label{fig:AffinityMatrixMarch}
\end{figure}
\begin{figure}[ht]
	\includegraphics[width=11cm]{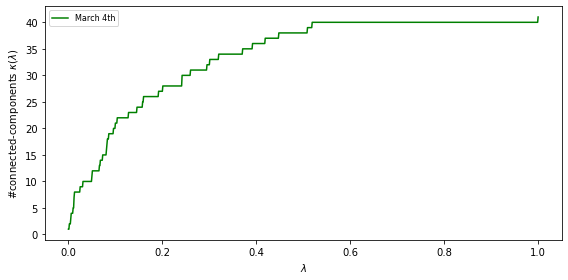}
	\caption{Connectivity curve $\kappa_1(\lambda)$ for affinity matrix $A^1$ (March)}\label{fig:AMBASUBEconnectedComponentsMarch}
\end{figure}
The matrix $A^2$, Figure~\ref{fig:AffinityMatrixApril}, is constructed in the same way with the data corresponding to April~8th 2020.
\begin{figure}[!]
	\includegraphics[width=11cm]{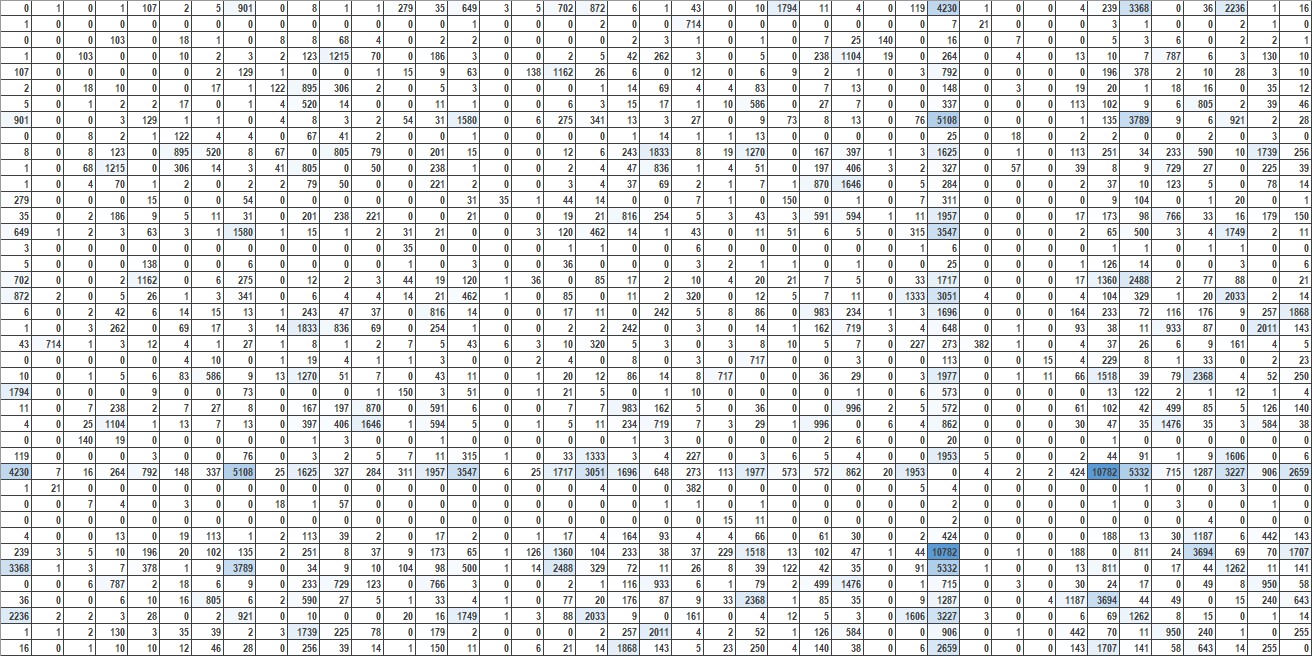}
	\caption{Unnormalized affinity matrix $A^2$ corresponding to April 2020}\label{fig:AffinityMatrixApril}
\end{figure}
\begin{figure}[!]
	\includegraphics[width=11cm]{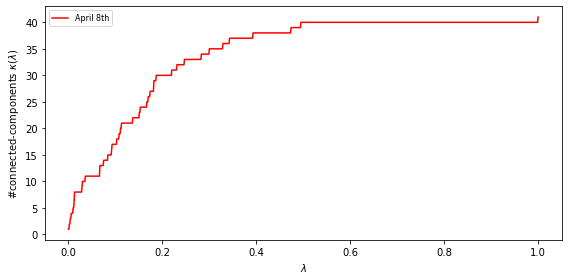}
	\caption{Connectivity curve $\kappa_2(\lambda)$ for affinity matrix $A^2$ (April)}\label{fig:AMBASUBEconnectedComponentsApril}
\end{figure}
The third $A^3$, Figure~\ref{fig:AffinityMatrixJune}, is based on the SUBE data by June~3rd 2020. Figure~\ref{fig:AMBASUBEconnectedComponentsMarch}, Figure~\ref{fig:AMBASUBEconnectedComponentsApril} and Figure~\ref{fig:AMBASUBEconnectedComponentsJune} show the shapes of the connectivity curves $\kappa_i(\lambda)$, $i=1,2,3$.
\begin{figure}[!]
	\includegraphics[width=11cm]{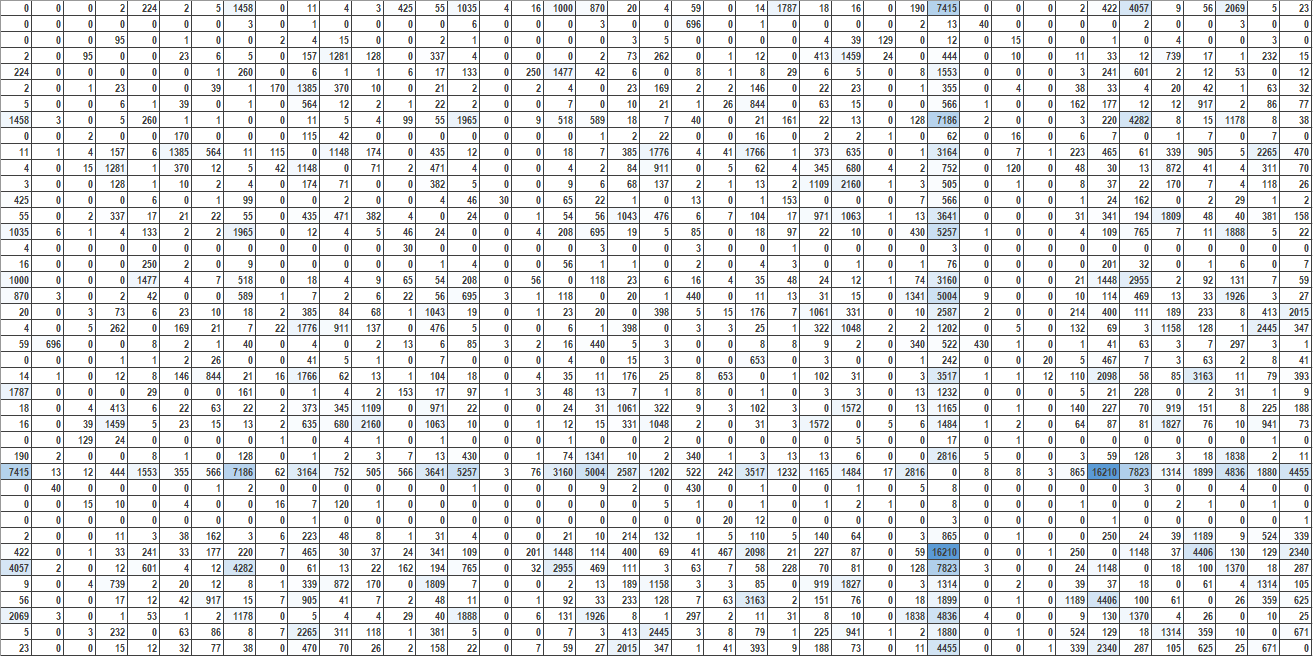}
	\caption{Unnormalized affinity matrix $A^3$ corresponding to June 2020}\label{fig:AffinityMatrixJune}
\end{figure}
\begin{figure}[!]
	\includegraphics[width=11cm]{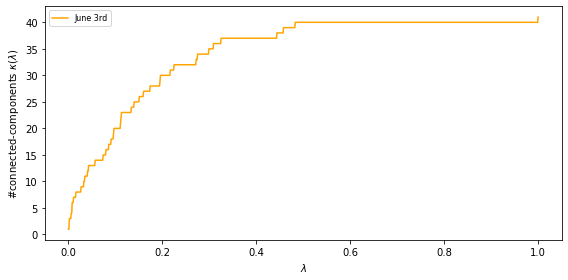}
	\caption{Connectivity curve $\kappa_3(\lambda)$ for affinity matrix $A^3$ (June)}\label{fig:AMBASUBEconnectedComponentsJune}
\end{figure}
The drastic changes of the  number of passengers between the three different dates, as could be expected do not have a great infect in the shapes of these three connectivity dynamics. 

Among the four distributions of data points shown in Figure~\ref{fig:allSeriesConnnectedComponents} the one that better fits the shapes of $\kappa_i$ is that of $X_3$. See Figure~\ref{fig:ConnectedComponentsMarchAprilJuneComparitions}.
\begin{figure}[]
	\includegraphics[width=11cm]{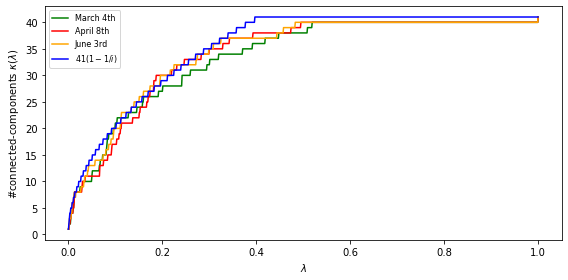}
	\caption{Connected components for SUBE affinity on March, April and June 2020 and the set $X_3$ in Section~\ref{sec:ThresholdingOfAffinities}.}\label{fig:ConnectedComponentsMarchAprilJuneComparitions}
\end{figure}

Let us finally observe that if we use affinities built only on geographic data, the behavior of the curves $\kappa(\lambda)$ instead of concave look roughly convex. In Figure~\ref{fig:ConnectedComponentsMarchW2LongitudFronterasCompartidas} the affinity between districts $i$ and $j$ is determined by the length of the shared boundaries. In Figure~\ref{fig:ConnectedComponentsMarchW3DistanciasEuclideaentreDistritos}, the affinity $A_{ij}=\frac{1}{d_{ij}}$, with $d_{ij}$ the Euclidean distance of the geographical centers of districts $i$ and $j$.
\begin{figure}[ht]
	\includegraphics[width=11cm]{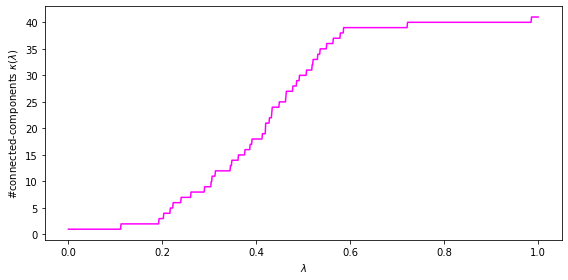}
	\caption{Connected components for matrix $A_{ij} = l_{ij}$ where $l_{ij}$ the length of the boundaries shared by the districts $i$ and $j$ of AMBA.}\label{fig:ConnectedComponentsMarchW2LongitudFronterasCompartidas}
\end{figure}
\begin{figure}[ht]
	\includegraphics[width=11cm]{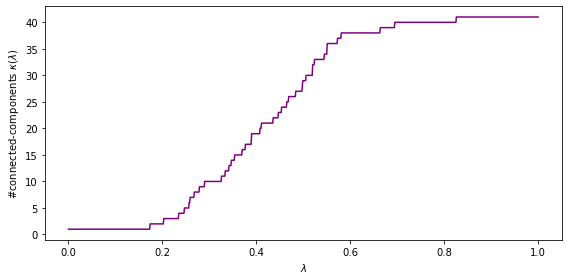}
	\caption{Connected components for matrix $A_{ij} = \frac{1}{d_{ij}}$ where $d_{ij}$ is the Euclidean distance between the geographical centers of the districts $i$ y $j$ in AMBA.}\label{fig:ConnectedComponentsMarchW3DistanciasEuclideaentreDistritos}
\end{figure}


\begin{thebibliography}{AAGM22}
	
	\bibitem[AAGM22]{AcAiGoMoTCAM22}
	M.~F. Acosta, H.~Aimar, I.~G\'{o}mez, and F.~Morana, \emph{Diffusive metrics
		induced by random affinities on graphs. {A}n application to the transport
		systems related to the {COVID}-19 setting for {B}uenos {A}ires ({AMBA})},
	Trends Comput. Appl. Math. \textbf{23} (2022), no.~4, 783--799. \MR{4537089}.
	DOI \href{https://doi.org/10.5540/tcam.2022.023.04.00783}{https://doi.org/10.5540/tcam.2022.023.04.00783}
	
	\bibitem[AG18]{AiGoAGMS18}
	Hugo Aimar and Ivana G\'{o}mez, \emph{Affinity and distance. {O}n the
		{N}ewtonian structure of some data kernels}, Anal. Geom. Metr. Spaces
	\textbf{6} (2018), no.~1, 89--95. \MR{3816950}. DOI \href{https://doi.org/10.1515/agms-2018-0005}{https://doi.org/10.1515/agms-2018-0005}
	
	\bibitem[CL06]{CoifmanLafon06}
	Ronald~R. Coifman and St\'{e}phane Lafon, \emph{Diffusion maps}, Appl. Comput.
	Harmon. Anal. \textbf{21} (2006), no.~1, 5--30. \MR{2238665}. DOI \href{https://doi.org/10.1016/j.acha.2006.04.006}{https://doi.org/10.1016/j.acha.2006.04.006}
	
	\bibitem[Kel62]{KelleybookEudeba}
	John~L. Kelley, \emph{Topolog\'{\i}a general}, Editorial Universitaria de
	Buenos Aires, Buenos Aires, 1962, Translated from the English and revised by
	Oscar A. Varsavsky. \MR{198408}
	
\end{thebibliography}

\providecommand{\bysame}{\leavevmode\hbox to3em{\hrulefill}\thinspace}
\providecommand{\MR}{\relax\ifhmode\unskip\space\fi MR }
\providecommand{\MRhref}[2]{%
	\href{http://www.ams.org/mathscinet-getitem?mr=#1}{#2}
}
\providecommand{\href}[2]{#2}


\section*{Statements and Declarations}
\subsection*{Funding}
Consejo Nacional de Investigaciones Cient\'ificas y T\'ecnicas, grant PIP-2021-2023-11220200101940CO.

\subsection*{Conflicts of interest/Competing interests}
The authors have no conflicts of interest to declare that are relevant to the content of this article.

\subsection*{Availability of data and material}
Not applicable.


\subsection*{Acknowledgements}
This work was supported by Consejo Nacional de Investigaciones Cient\'ificas y T\'ecnicas-CONICET and Universidad Nacional del Litoral-UNL, in Argentina.



\smallskip

\noindent{\footnotesize \textit{Affiliation.} 
	\textsc{Instituto de Matem\'{a}tica Aplicada del Litoral ``Dra. Eleonor Harboure'', CONICET, UNL.}
	
	%
%
%

\noindent \textit{Address.} \textmd{IMAL, Streets F.~Leloir and A.P.~Calder\'on, CCT CONICET Santa Fe, Predio ``Alberto Cassano'', Colectora Ruta Nac.~168 km~0, Paraje El Pozo, S3007ABA Santa Fe, Argentina.}

%
\noindent \textit{E-mail:} \verb|haimar@santafe-conicet.gov.ar| \\ \hspace*{1.2cm} \verb|carias@santafe-conicet.gov.ar|\\ \hspace*{1.2cm} \verb|ivanagomez@santafe-conicet.gov.ar| (corresponding author)
}
\end{document}